\documentclass[12pt]{amsart}

\newtheorem{theorem}{Theorem}

\newtheorem{conjecture}{Conjecture}
\newtheorem{lemma}{Lemma}

\newtheorem{definition}{Definition}

\begin{document}
\author{Gao Mou}\thanks{School of Physical and Mathematical Sciences, Nanyang Technological University, Singapore. Email: gaom0002@e.ntu.edu.sg}
\title{$2$-walks in $2$-tough $2K_2$-free graphs}
\begin{abstract}
We prove that every 2-tough $2K_2$-free graph admits a 2-walk.
\end{abstract}

\maketitle

\section{Introduction}

Let $G$ be a graph with vertex set  $V(G)$ and edge set $E(G)$. For any $v\in
V(G)$, let $N_G(v)$ denote the set of neighbors of $v$ in $G$. A $k$-walk of
$G$ is a spanning closed walk of $G$ visiting each vertex at most $k$ times. An
$1$-walk of $G$ is a Hamiltonian cycle in $G$. For an integer $m$, denote by $m*G$ the multigraph obtained from $G$ by taking each edge $m$ times. 
Obviously, a $k$-walk in $G$, is also a subgraph of $(2k)*G$.
Let $\Omega(G)$ denote the number of connected components of $G$. 
The following terminology, due to V.~Chv\'{a}tal \cite{chvatal1973tough}, turned out to be very important in the research of Hamiltonicity.
\begin{definition}
$G$ is $\beta$-{\em tough},
for a positive real $\beta$, if $\Omega(G-S)>1$ implies $|S|\ge \beta\cdot \Omega(G-S)$ 
for each $S\subset G$.
\end{definition}
That is, $G$ is $\beta$-tough if
$G$ cannot be split into $k$ (with $k>1$) components by removing less than
$k\beta$ vertices.  The {\em toughness} of $G$, denoted $\tau(G)$, is the maximum
value of $\beta$ for which $G$ is $\beta$-tough.  Clearly, if $g$ is
Hamiltonian, then $G$ is 1-tough, however, the converse is not true.  A famous
conjecture of V.~Chv\'{a}tal \cite{chvatal1973tough}, which is still open, claims
that the converse holds at least in an approximate sense.
\begin{conjecture}\label{conj1} There exists a constant $\beta$ such that every
$\beta$-tough graph is Hamiltonian.  \end{conjecture}

The concept of a $k$-walk is a generalization of the concept of a Hamiltonian cycle; in \cite{jackson1990k} B.~Jackson and N.C.~Wormald investigated $k$-walks and obtained the following results.
\begin{theorem}\label{thm1}
Every $1/(k-2)$-tough graph has a $k$-walk. 
In particular every 1-tough graph has a 3-walk. \qed
\end{theorem}

In \cite{ellingham2000toughness}, it is proved that
\begin{theorem}\label{thm3}
Every 4-tough graph has a 2-walk. \qed
\end{theorem}
The following well-known conjecture related to $k$-walks appeared in \cite{jackson1990k}.
\begin{conjecture}\label{conj2}
Every $1/(k-1)$-tough graph has a $k$-walk.
\end{conjecture}

Results just mentioned do not apply to the case $k=1$.
For some classes of graphs, there are strong results connected toughness and Hamiltonicity (recall that 
a 1-walk is a Hamiltonian cycle). E.g. in \cite {chen1998tough}, G.~Chen, M.S.~Jackson, A.E.~Kezdy, and J.~Lehel proved
\begin{theorem}\label{thm4}
Every 18-tough {\em chordal} graph is Hamiltonian. \qed
\end{theorem}

\begin{definition}
$G$ is said to be {\em split} if $V(G)$ can be partitioned into an independent set $I$ and a clique $C$.
\end{definition}
For split graphs, we have many beautiful results.
E.g. in \cite{kratsch1996toughness} the following is proved.
\begin{theorem}\label{thm5}
Every $3/2$-tough {\em split} graph on at least three vertices is Hamiltonian, and this is best possible in the sense that there is a sequence $\{G_n\}^{\infty}_{n=1}$ of split graphs with no 2-factor and $\tau(G_n)\to 3/2$. \qed
\end{theorem}

Let us consider a superclass of split graphs, named $2K_2$-free graphs. These are graphs which do not contain an induced copy of $2K_2$, the graph on four vertices consisting of two vertex disjoint edges. Obviously, every split graph is a $2K_2$-free graph. What is more, every co-chordal graph, i.e. the complement of a chordal graph, is also a $2K_2$-free graph. That means that the class of $2K_2$-free graphs is as rich as the class of chordal graphs.

Recently, in \cite{broersma2014toughness}, H.~Broersma, V.~Patel and A.~Pyatkin proved the following. 
\begin{theorem}\label{thm6}
Every 25-tough 2$K_2$-free graph on at least three vertices is Hamiltonian. \qed
\end{theorem}

In this paper, we prove
\begin{theorem}\label{thm7}
Every 2-tough 2$K_2$-free graph admits a 2-walk.
\end{theorem}

\section{On 2$K_2$-free graphs}

We present several structural properties of $2K_2$-free graphs which turn out to be very useful in the proof of the main theorem.
For a subset $A\subset V(G)$, let $Dom(X)$ denote the set of vertices dominated by $A$, i.e. $Dom(A)=A\cup\{y\in V(G);~there~exists~x\in A~such~that~xy\in E(G)\}$. The set $A\subset V(G)$ is said to be {\em dominating} if $Dom(A)=V(G)$. A {\em dominating clique} of a graph $G$ is a dominating set which induces a complete subgraph in $G$. The following theorem comes from \cite{chung1990maximum}.

\begin{theorem}\label{thm8}
If $G$ is $2K_2$-free and the maximum size of cliques $\omega(G)\ge3$, then $G$ has a dominating clique of size $\omega(G)$. \qed
\end{theorem}
Let us consider a generalization of dominating set. The set $A\subset V(G)$ is said to be {\em weakly-dominating} if for any edge $v_1v_2\in G$, we have $v_1\in Dom(A)$ or $v_2\in Dom(A)$. And, a {\em weakly-dominating clique} of a graph $G$ is a weakly-dominating set which induces a complete subgraph in $G$. We say a clique $Q_j$ of $G$ is weakly-dominated by a clique $Q_i$ of $G$, if for any pair $(v_1,v_2)$ of vertices in $Q_j$ one has that at least one of $v_1$ and $v_2$ is adjacent to a vertex in $Q_i$.

Obviously, a dominating set is a weakly-dominating set.
Then, similarly to Theorem \ref{thm8}, we get:
\begin{theorem}\label{weak}
If $G$ is $2K_2$-free and the maximum size of clique $\omega(G)\ge2$, then $G$ has a weakly-dominating clique of size $\omega(G)$.
\end{theorem}
\begin{proof}
If $\omega(G)\ge3$, by Theorem \ref{thm8}, we get a dominating clique, thus it is also a weakly-dominating clique. If $\omega(G)=2$, for any clique of size 2, say $Q_0=v_1v_2$ and any edge $v_3v_4 \in G$, if $v_1v_3,v_1v_4,v_2v_3,v_2v_4\not\in E(G)$, then $v_1v_2$ and $v_3v_4$ form a $2K_2$.
\end{proof}

In fact, in a $2K_2$-free graph $G$, any edge is weakly-dominating.
The following observation from \cite{broersma2014toughness} is very useful.
\begin{lemma}\label{lm1}
A graph $G=(V,E)$ is $2K_2$-free if and only if for every $A\subset V$ at most one component of the graph $G-A$ contains edges. \qed
\end{lemma}
Using these two properties as tools, we can look at $2K_2$-free graphs more closely.

Given a $2K_2$-free graph $G$, by Theorem \ref{thm8}, we can find one of its maximum weakly-dominating clique, namely $Q_1$. Obviously, any induced subgraph of a $2K_2$-free graph is again $2K_2$-free. Then $G-Q_1$ is also a $2K_2$-free graph. By Lemma 1, $G-Q_1$ is made up by two parts, one is an independent subset (possibly empty) of $G-Q_1$, denoted by $D_1$, another part is a non-trivial component (possibly empty), denoted by $G_1$, which is also $2K_2$-free. For the same reason, we can find a maximum weakly-dominating clique in $G_1$, namely $Q_2$, and a non-trivial component in $G_1-Q_2$, namely $G_2$ and an independent subset in $G_1-Q_2$, namely $D_2$. Repeating this process, we get
\begin{theorem}\label{thm9}
For a $2K_2$-free graph $G=G_0$, we can find a sequence of cliques $\{Q_i;i=1,\ldots,k\}$, where $|Q_i|\ge2$ and $|Q_i|\ge|Q_{i+1}|$, such that $Q_1$ is a maximum weakly-dominating clique in $G_0$ and $Q_{i+1}$ is a maximum weakly-dominating clique in $G_{i}\subset G_{i-1}-Q_{i} ,(i=1,\ldots,k-1)$. Additionally, $G_i$ is the only non-trivial component in $G_{i-1}-Q_i$. The subset $D_i=G_{i-1}-Q_i-G_i$ is an independent set. We call vertices in $D=\cup_{i=1,\ldots,k} D_i$ the first class vertices.  \qed
\end{theorem}

In addition, we get
\begin{theorem}\label{thm10}
A $2K_2$-free graph $G$ can be divided into two parts, a ``clique tower'' $Q=\cup_{i=1,\ldots,k}Q_i$ and an independent set $D=\cup_{i=1,\ldots,k}D_i$, where $V(G)=V(Q)\cup V(D)$ and $V(Q)\cap V(D)=\emptyset$. Note that we allow any of these sets to be empty. \qed
\end{theorem}

\section{The proof of the main result}

In short, the proof is divided into two parts. In the first part, we construct an auxiliary graph $\Gamma$, which is an Eulerian (multi)graph. There are two kinds of edges in $\Gamma$, namely blue edges and red edges. In the second part, with the help of $\Gamma$, we find a subgraph $H$ of $2*G$, which is Eulerian and with all vertices of degree 2 or 4, where $2*G$ denotes the multigraph obtained by doubling each edge of $G$ into a pair of parallel edges. $H$ is the 2-walk we want. In $H$, there are three kinds of edges, namely first-class edges, second-class-edges and third-class edges. The first-class edges are corresponding to the blue edges in $\Gamma$, the second-class edges are corresponding to red edges in $\Gamma$, and the third-class edges are used to make sure $H$ is connected and to adjust the vertex degrees to 2 or 4.

\subsection{The proof of Theorem \ref{thm7}}

Let $G$ be a 2-tough $2K_2$-free graph. If it has only 2 vertices, this is a trivial case and nothing to prove. So, we assume there are at least three vertices. The sequence of cliques defined in Theorem \ref{thm9} is $\{Q_1,\ldots,Q_k\}$, where each $Q_i$ is weakly-dominated by $Q_j$, when $j<i$, and $|Q_i|\ge|Q_{i+1}|\ge2$.

First, if $|Q_1|=2$, then $G$ is triangle-free, then, by \cite[Theorem 4]{broersma2014toughness}, $G$ is Hamiltonian, and thus $G$ has a 2-walk.
Additionally, if there is only one clique $Q_1$ in the sequence, i.e. $k=1$, then $G$ is split graph. By \cite[Theorem 3.3]{ kratsch1996toughness}, $G$ is Hamiltonian, and thus has a 2-walk.

Now, we assume $k\ge2$ and $|Q_1|\ge3$. The independent set (also called the
first-class vertex set) is denoted by $D$, as in Theorem \ref{thm9}.  Let $D_0\subseteq D$. 
By 2-toughness, the size of the neighbor
set $N_G(D_0)$, of $D_0$ in $Q$ is at least $2|D_0|$, i.e.
$|N_G(D_0)|\ge2|D_0|$, otherwise, after deleting $N(D_0)$, there are at least
$|D_0|$ components (isolated vertices), since $D_0$ is an independent set.  By
the polygamous form of Hall theorem, there is a subset $Q'$ of $Q$,
$|Q'|=2|D|$, and each vertex in $D$ is adjacent to two distinct vertices in
$Q'$. That means there is a subset $E'\subset E(G)$ where each $e\in E'$ is
connected to a vertex in $D$ and a vertex in $Q'$. Moreover, each vertex in
$Q'$ is incident to exactly one edge in $E'$, and each vertex in $D$ is
incident to exactly two edges in $E'$. We call the edges in $E'$ the {\bf
first-class edges} in $G$, (also in $H$, we will see that later).

\subsubsection*{The construction of $\Gamma$}

Now, let us construct the auxiliary graph $\Gamma$.  First, the vertex set of
$\Gamma$, $V(\Gamma)=\{w_1,\ldots,w_k\}\cup D'$, each $w_i$ corresponds to the
clique $Q_i$, and each $v'_j\in D'$ corresponds to  $v_j\in D$.  For each
first-class edge $e\in E(G)$, we draw the corresponding edge $e'$ on $\Gamma$,
i.e., if $e=v_iv_j\in E(G)$, with $v_i\in D$ and $v_j\in Q_t$, then we have
$v'_iw_t\in E(\Gamma)$. We call it a {\bf blue edge} in $\Gamma$. Note
that we allow parallel edges in $\Gamma$, and degrees of
vertices count parallel edges with multiplicity.

After drawing the blue edges on $\Gamma$, let us look at the components of
$\Gamma$ (some of them may be trivial components, i.e. single points).
Obviously, each component has even number of vertices with odd degree.

\subsubsection*{Case 1} 
If there is only one component, then pair up all the
odd degree vertices by adding edges, that means drawing a maximum matching
between these odd degree vertices. And the edges in the matching are called the
{\bf red edges}. Then $\Gamma$ becomes an Eulerian graph, with each vertex
incident to at most one red edge. And note that only vertices $\{w_i\}$ are
possibly incident to a red edge.

\subsubsection*{Case 2} 
If there are at least two components, say $C_1,\ldots,C_n$, where $n\ge2$,
let us select some {\bf representative vertices} in each of them. 
For any component, say $C_i$,
with some odd degree vertices, select two odd degree vertices, denoted by
$v_i^+$, $v_i^-$, from them as representative vertices. Note that they
are not in $D'$, since all vertices in $D'$ have degree 2. On the other hand,
for any component, say $C_j$ with only even degree vertices, select one vertex,
denoted by $v_j$, as representative vertex. Note that we can require
$v_j\not\in D'$, since each component has at least one vertex $w_t$. For
convenience, we denote $v_j^-=v_j^+=v_j$ in this case. Then, we draw
$v_1^+v_2^-,~v_2^+v_3^-,\ldots,v_{n-1}^+v_n^-,~v_n^+v_1^-$ on $\Gamma$ (if
$n=2$ and both components have only even degree vertices, this circle is a pair
of parallel edges $v_1^+v_2^-,v_2^+v_1^-$), and we also call them the {\bf red
edges}. For the odd degree vertices which are not selected as representative
vertices, we pair them up within their components, that means drawing a
matching on these vertices with all these matching edges do not cross different
components. These matching edges are also called the {\bf red edges}.

Checking {\bf Cases 1 and 2}, we find that for each vertex $w_i$ one
of the following holds.

\begin{enumerate}
\item $w_i$ is not incident to any red edges.
\item $w_i$ is incident to one red edge as representative vertex.
\item $w_i$ is incident to two red edges as representative vertex.
\item $w_i$ is incident to one red edge for pairing up.
\end{enumerate}

Now, $\Gamma$ is an Eulerian (multi)graph, since it is a connected graph with
all vertices of even degree. Note that all red edges are only incident
to vertices $\{w_i\}$, and each vertex in $D'$ is incident to exactly two blue
edges.

\subsubsection*{The construction of $H$}
Now with the help of $\Gamma$ we are going to find a spanning subgraph $H$ of $2*G$.

First, the vertex set of $H$, is set to be the vertex set of $G$. So, we use the same notation for these vertices in $H$ and $G$. Additionally, when we talk about any $Q_i$ in $G$, we mean the clique $Q_i$, on the other hand, when we talk about $Q_i$ in $H$, we are talking about the subset of vertices.

{\bf Step 1.} Add all the first-class edges connecting $Q'$ and $D$ to $H$. That means the first-class edges in $G$ and the first-class edges in $H$ are exactly the same set.

{\bf Step 2.} (Finding second-class edges in $H$.)

For each red edge $w_iw_j\in\Gamma$, (we assume $i<j$), we want to find an edge
in $G$, with one endpoint in $Q_i$ and another in $Q_j$. This is quite an easy
task. By Theorem \ref{thm9}, $Q_j$ is weakly-dominated by $Q_i$, that means,
in $G$, some vertices in $Q_j$ are adjacent to some vertices in $Q_i$. Now, we
can arbitrarily select one edge, with one endpoint in $Q_j$ while another in
$Q_i$, from $G$ into $H$. Such edges are called {\bf second-class edges}.

{\bf Step 3.} (Finding third-class edges in $H$.)
The final step in the construction is to find some so-called {\bf third-class
edges} for $H$. After adding all the first-class edges and all the
second-class edges, let us look at each $Q_i$.

In any $Q_i$, the sum of vertex degrees (corresponding to first-class and
second-class edges) is obviously even, because $\Gamma$ is an Eulerian graph,
and the sum of vertex degrees in $Q_i$ equals the degree of $w_i$ in $\Gamma$.
So, in $Q_i$, there must be even number of odd degree vertices, if any. 

In $Q_i$ there is an even number (maybe 0) of odd degree (in $H$) vertices, and by
properties of $\Gamma$ and by construction of $H$ carried so far at most one of
these vertices, say $v_x$, has degree 3, while the rest of them have degree 1.
Now, select a maximum matching, of {\bf third-class edges}, on them. Now $v_x$,
if it exist at all, has degree 4, and the remaining vertices in $Q_i$ have
degree 0 or 2.

Now we connect all the vertices in $Q_i$ (except $v_x$, if it exists), by a circle of {\bf third-class edges}.
More precisely, if $v_x$ is present, and $|Q_i|=2$, we do nothing, and in case we need to connect
just two vertices by a circle we use a pair of parallel edges.

After drawing all these first-, second-, and third-class edges on $H$, we observe that:
\begin{enumerate}
\item $H$ is connected multigraph.
\item $H$ is a subgraph of $2*G$.
\item each vertex in $H$ has degree 2 or 4, where the degree of parallel edges is counted with multiplicity.
\end{enumerate}

We conclude that $H$ is Eulerian, and  $H$ is a 2-walk in $G$. \qed

\medskip
 
We conclude this paper  with the following
\begin{conjecture}
Every 2-tough $2K_2$-free graph on at least three vertices has a 2-trail,
i.e.  a 2-walk with each edge appearing in the walk at most once.
\end{conjecture}

\bibliography{reftough}
\bibliographystyle{abbrv}
\end{document}